\documentclass[12pt]{amsart}
\usepackage{amssymb,amsfonts,amsmath,amsopn,amstext,amscd,latexsym,amsthm,enumerate,mathrsfs,bbm,color}

\usepackage[margin=36mm]{geometry}
\usepackage[all,cmtip]{xy}
\usepackage{longtable}

\newtheorem{theorem}{Theorem}[section]
\newtheorem{thm}[theorem]{Theorem}

\newtheorem{lem}[theorem]{Lemma}

\newtheorem{cor}[theorem]{Corollary}

\newtheorem{exam}[theorem]{Example}

\newtheorem{introthm}{Theorem}

\theoremstyle{remark}

\numberwithin{equation}{section}

     \DeclareMathOperator{\IBr}{IBr}

 \DeclareMathOperator{\irr}{Irr}

\begin{document}

\title[Gallagher's theorem]
{A converse for a theorem of Gallagher }

\author[Xiaoyou Chen]{Xiaoyou Chen}
\address{School of Mathematics and Statistics, Henan University of Technology, Zhengzhou 450001, China}
\email{cxymathematics@hotmail.com}

\author[Mark L. Lewis]{Mark L. Lewis}
\address{Department of Mathematical Sciences, Kent State University, Kent, OH 44242, USA}
\email{lewis@math.kent.edu}


\date{\today}


\begin{abstract}
Let $G$ be a finite group.  Suppose $N$ is a normal subgroup of $G$.  Recall that Gallagher's theorem states that if $\chi \in {\rm Irr} (G)$ satisfies $\chi_N$ is irreducible, then $\chi \beta$ is irreducible and distinct for all $\beta \in {\rm Irr} (G/N)$.  Furthermore, if $\theta = \chi_N$, then these are all of the irreducible constituents of $\theta^G$.  We prove that the converse of this theorem holds.  We also prove that a partial converse of the Brauer version of this theorem holds.  Finally, we prove that an analog of Gallagher's theorem holds for Isaacs' $\pi$-partial characters and that a partial converse of that theorem is true.
\end{abstract}

\maketitle

\centerline{\it In memory of I. M. Isaacs }

\section{Introduction}

In this paper, all groups are finite.  Take $G$ to be a group, $p$ to be a prime, and $\pi$ to be a set of primes. We write ${\rm Irr} (G)$ for the set of irreducible characters of $G$, ${\rm IBr} (G)$ for the set irreducible ($p$-)Brauer characters of $G$, and ${\rm I}_\pi (G)$ for the set of irreducible $\pi$-partial characters of $G$.

Gallagher's theorem \cite[Corollary 6.17]{Isaacs1976} can be stated as follows: let $N$ be a normal subgroup of $G$ and let $\chi\in {\rm Irr} (G)$ be such that $\chi_{N} = \theta \in {\rm Irr} (N)$, where $\chi_{N}$ is the restriction of $\chi$ to $N$. Then the characters $\beta\chi$ for $\beta\in {\rm Irr} (G/N)$ are irreducible and distinct for distinct $\beta$ and are all of the irreducible constituents of $\theta^{G}$.

The first goal in this paper is to prove the converse of this theorem:

\begin{introthm} \label{main1}
Let $G$ be a group and let $N$ be a normal subgroup of $G$.  Suppose ${\rm Irr} (G/N)=\{\beta_{1}, \cdots, \beta_{n}\}$ and $\chi\in {\rm Irr} (G)$, and suppose that the $\beta_{1}\chi, \cdots, \beta_{n}\chi$ are distinct and irreducible.  Then $\chi_{N}$ is an irreducible character of $N$.
\end{introthm}

We will see that we will obtain Theorem \ref{main1} as a corollary to the version for the Brauer character version of the theorem.  The Brauer character version of Gallagher's theorem \cite[Corollary 8.20]{Navarro1998} is as follows: 
Let $N\lhd G$ and let $\eta\in \IBr (G)$. If $\eta_{N} = \theta \in \IBr (N)$, then the characters $\beta\eta$ for $\beta\in \IBr(G/N)$ are irreducible, distinct for distinct $\beta$ and are all of the irreducible constituents of $\theta^{G}$.  Now we consider a partial converse of the Brauer character version of Gallagher's theorem.

\begin{introthm} \label{main2}
Let $G$ be a group, let $N$ be a normal subgroup of $G$, and let $p$ be a prime.  Suppose $\eta\in {\rm IBr}(G)$ and ${\rm IBr}(G/N)=\{\beta_{1}, \cdots, \beta_{n}\}$, and suppose that $p\nmid |G/N|$ and the $\beta_{1}\eta, \cdots, \beta_{n}\eta$ are distinct and irreducible. Then $\eta_{N}$ is an irreducible Brauer character of $N$.
\end{introthm}

We next consider Isaacs' $\pi$-theory, where $\pi$ is a set of primes.  For a $\pi$-separable group $G$, Isaacs has defined an analog of Brauer characters that are defined on the $\pi$-elements of $G$ that are called the $\pi$-partial characters of $G$.  We will write ${\rm I}_\pi (G)$ for the set of irreducible $\pi$-partial characters of $G$.  We will review the points of $\pi$-theory we need in Sections \ref{secn: pi theory} and \ref{secn: nucl}.  Interestingly, while Isaacs proved analogs of many other results for $\pi$-partial characters, he does not seem to have proved an analog of Gallagher's theorem.  We do that next.

\begin{introthm}\label{ipi Gall,gen}
Let $\pi$ be a set of primes, let $G$ be a $\pi$-separable group, and let $N$ be a normal subgroup of $G$.  Suppose there exists a character $\zeta \in {\rm I}_{\pi} (G)$ so that $\xi = \zeta_N \in {\rm I}_\pi (N)$.  Then the map $\kappa \mapsto \kappa\zeta$ is a bijection from ${\rm I}_{\pi} (G/N)$ to ${\rm I}_\pi (G \mid \xi)$.
\end{introthm}

Now, in \cite{Isaacs1976}, there is a generalization of Gallagher's theorem (Theorem 6.16 of \cite{Isaacs1976}) of which Gallaher's theorem is obtained as a corollary.  We have not been able to obtain an analog of this generalization in cases, but we can obtain an analog when we add the extra assumption that $2 \in \pi$ or $|G|$ is odd.

\begin{introthm} \label{ipi pre,2inpi,|G|odd}
Let $\pi$ be a set of primes, let $G$ be a $\pi$-separable group, and let $N$ be a normal subgroup of $G$.  Suppose $2 \in \pi$ or  $|G|$ is odd.  Assume there exist partial characters $\eta, \xi \in {\rm I}_\pi (N)$ so that $\eta$ is $G$-invariant, $\xi$ extends to $\zeta \in {\rm I}_\pi (G)$ and $\eta \xi \in {\rm I}_{\pi} (N)$.  Then the map $\kappa \mapsto \kappa \zeta$ is a bijection from ${\rm I}_\pi (G \mid \eta) \rightarrow {\rm I}_\pi (G \mid \eta \xi)$.
\end{introthm}

We close with a partial converse for the analog of Gallagher's theorem for $\pi$-partial characters.  We will show that it is not possible to have a full converse in this case.

\begin{introthm}\label{ipi conv}
Let $\pi$ be a set of primes, let $G$ be a $\pi$-separable group, and let $N$ be a normal subgroup of $G$.
Suppose there exists partial characters $\zeta \in {\rm I}_{\pi} (G)$ and ${\rm I}_{\pi} (G/N) = \{\kappa_1, \dots, \kappa_n \}$.  Assume that $G/N$ is a $\pi$-group and the $\kappa_1 \zeta, \dots, \kappa_n \zeta $ are distinct and irreducible  Then $\zeta_N$ is irreducible.
\end{introthm}

\section{Brauer Characters and ordinary characters}

In this section, we prove Theorems \ref{main1} and \ref{main2}.  We begin with the Brauer character version which is Theorem \ref{main2}.

\begin{proof}[Proof of Theorem \ref{main2}]
Let 
$\theta\in {\rm IBr}(N)$ be an irreducible constituent of $\eta_{N}$, where $\eta_{N}$ is the restriction of $\eta$ to the set of $p$-regular elements of $N$. Then by Clifford's theorem \cite[Corollary 8.7]{Navarro1998} we have that $\eta_{N}=e\sum_{i=1}^{t}\theta_{i}$, where $\theta_{1}=\theta$, $I_{G}(\theta)=\{g\in G\mid \theta^{g}=\theta\}$ is the inertia subgroup of $\theta$ in $G$, $t=|G: I_{G}(\theta)|$, and $e$ is a positive integer. And so $\eta(1)=et\theta(1)$.

Now let us consider the irreducible constituents of $\theta^{G}$. By Nakayama's theorem \cite[Lemma 8.4]{Navarro1998}, and since $p \nmid |G/N|$, we have that
$${\bf I}(\beta\eta,\theta^{G})={\bf I}((\beta\eta)_{N},\theta) ={\bf I}(\beta(1)\eta_{N}, \theta)=\beta(1){\bf I}(\eta_{N}, \theta)=\beta(1)e,$$
where $\beta\in {\rm IBr}(G/N)$ and ${\bf I}(\beta\eta,\theta^{G})$ is the multiplicity of $\beta\eta$ in $\theta^{G}$.  Thus, all the distinct $\beta\eta$ appear with multiplicity $\beta(1)e$ as irreducible constituents of $\theta^{G}$.

Therefore,
\begin{eqnarray*}
\theta^{G}(1)=\theta(1)|G: N|
&\geq& \sum_{\beta}\beta(1)e(\beta\eta)(1)\\
&=&e\eta(1)\sum_{\beta}\beta(1)^{2} =e\eta(1)|G:
N|=e^{2}t\theta(1)|G: N|.
\end{eqnarray*}
It follows that $e^{2}t\leq 1$, meaning $e=1=t$. Then $\eta_{N}=\theta$ and the proof is complete.
\end{proof}

We now obtain the ordinary character version, which is Theorem \ref{main1}, as a corollary to Theorem \ref{main2}.  The point here is that the irreducible characters of $G$ can be thought of Brauer characters for a prime $p$ that does not divide $|G|$.

\begin{proof}[Proof of Theorem \ref{main1}]
Take $p$ to be a prime so that $p\nmid |G|$. Then we know that ${\rm Irr} (G) = {\rm IBr} (G)$.  Since $p$ does not divide $|G:N|$, we have the result by applying Theorem \ref{main2} to $\chi$ viewed in ${\rm IBr} (G)$.
\end{proof}

The following example shows that the condition $p\nmid |G/N|$ cannot
be dropped from the hypotheses of Theorem \ref{main2}.

\begin{exam} \rm
Let $G=S_{4}$ and $N=K_{4}$ the Klein group, $p=3$. Then $G/N \cong S_{3}$ and ${\rm IBr} (S_{3}) = \{1, \lambda \mid \lambda(1) = 1 \}$. Let $\eta \in {\rm IBr}(G)$ with $\eta(1)=3$. Then $\{\eta, \eta\lambda\}$ are distinct and irreducible. However, $\eta_{N}$ is reducible.
\end{exam}

\section{$\pi$-theory} \label{secn: pi theory}

In the following three sections, we consider the $\pi$-theory version of Gallagher's theorem and its converse.  Isaacs introduces ``$\pi$-theory'' as a way of generalizing Brauer characters from a single prime $p$ to a set of primes $\pi$.  Unfortunately, Isaacs' $\pi$-theory is only defined for $\pi$-separable groups as opposed to Brauer characters which are defined for all groups.

The $\pi$-partial characters are first defined in \cite{pisep}, but many of the most important definitions and results have been included in the recent monograph \cite{Isaacs2018}.  Over the years, Isaacs also wrote a number of expository papers that present the $\pi$-theory results, and we recommend that the reader consult them.  They are \cite{arc}, \cite{pipart}, \cite{aust}, and \cite{nato}.  In reading those papers, one will see that Isaacs has proved for $\pi$-partial characters analogs of many of the results that have been proved for Brauer characters, especially Brauer characters of solvable or $p$-solvable groups.  In particular, many of the results for ordinary characters that have analogs for Brauer character; Isaacs proves that there are analogs for $\pi$-partial characters.  Interestingly, even though there is a Brauer analog for Gallagher's theorem, Isaacs does not have an analog of Gallagher's theorem for $\pi$-partial characters.

In \cite{Isaacs1976}, there is an extension of Gallagher's theorem, and in Section \ref{secn: galla}, we will show that we can prove an analog of this extension for $\pi$-partial characters when $2 \in \pi$ or $|G|$ is odd, and then obtain an analog of Gallagher's theorem in this case.  However, the technique used to prove this generalization does not seem to work in the general case.  In Section \ref{secn: gen}, we will use a different technique to prove an analog of Gallagher's theorem for $\pi$-partial characters for all sets of primes $\pi$ and all $\pi$-separable groups.

To define the $\pi$-partial characters of $G$, we are going to follow the approach in \cite{pipart}.  We take $\pi$ to be a set of primes and $G$ to be a $\pi$-separable group.  We set $G^o$ to be the set of $\pi$-elements of $G$.  When $\chi$ is a character of $G$, we define $\chi^o$ to be the restriction of $\chi$ to $G^o$.  We define the set of $\pi$-partial characters of $G$ to be the set of restrictions of characters of $G$ to $G^o$.  It is not difficult to see that the set of $\pi$-partial characters of $G$ are closed under sums.  Define ${\rm I}_\pi (G)$ to be those $\pi$-partial characters that cannot be written as a sum of two other $\pi$-partial characters.  We say that these are the {\it irreducible $\pi$-partial characters} of $G$.  Isaacs shows that ${\rm I}_\pi (G)$ forms a basis for the vector space of complex valued class functions defined on $G^o$.

Given an irreducible $\pi$-partial character $\phi \in {\rm I}_\pi (G)$, we say that an irreducible character $\chi \in \irr (G)$ is a {\it lift} of $\phi$ if $\chi^o = \phi$.   By the way we have defined ${\rm I}_\pi (G)$, it is not difficult to see that every partial character in ${\rm I}_\pi (G)$ necessarily has a lift.  Now, Isaacs has produced a ``canonical'' set of lifts and developed a number of properties of this canonical set of lifts, and using these properties, he proves many of the properties of the $\pi$-partial characters.  He calls this canonical set of lifts the ${\rm B}_\pi$-characters.  We note that other sets of lifts have been constructed in  \cite{dpi}, \cite{nav}, \cite{nucl}, and \cite{cos}.  In fact, we will use the set of lifts in \cite{nucl} to prove the general version of our analog of Gallagher's theorem, and we will discuss those lifts in Section \ref{secn: nucl}.

The ${\rm B}_\pi$-characters are initially defined in \cite{pisep} and one can find the definition and main results in \cite{Isaacs2018}.  Due to the complications of the definition, we are not going to state it here.  We will just state some of the main properties that we need, and give references to the proofs.  All of the results in this paragraph are initially proved in \cite{pisep}, but we will give references for the proofs in \cite{Isaacs2018}. For a $\pi$-separable group $G$, the set ${\rm B}_{\pi} (G)$ will be a subset of $\irr (G)$.  The map $\chi \mapsto \chi^o$ will be a bijection from ${\rm B}_\pi (G) \rightarrow {\rm I}_\pi (G)$ (Theorem 5.1 of \cite{Isaacs2018}).  If $N$ is a normal subgroup of $G$ and $\chi \in {\rm B}_\pi (G)$, then the irreducible constituents of $\chi_N$ lie in ${\rm B}_\pi (N)$ (Corollary 4.21 (a) of \cite{Isaacs2018}).  If $N$ is a normal subgroup so that $G/N$ is a $\pi$-group and $\theta \in {\rm B}_\pi (N)$, then all of the irreducible constituents of $\theta^G$ lie in ${\rm B}_\pi (G)$ (Theorem 4.14 of \cite{Isaacs2018}).  When $G/N$ is a $\pi'$-group and $\theta \in {\rm B}_\pi (N)$, then $\theta^G$ has a unique irreducible constituent in ${\rm B}_\pi (G)$ (Theorem 4.19 (b) of \cite{Isaacs2018}).

In \cite{odd}, Isaacs shows that when $2 \in \pi$ or $|G|$ is odd, then there are a number of additional properties that make the ${\rm B}_\pi$-characters easier to work with.  We will use these additional properties to show that an analog for the generalization of Gallagher's theorem holds under these hypotheses for $\pi$-partial characters.   The results we state in this paragraph are all initially proved in odd, but Isaacs has given another treatment in \cite{Isaacs2018}, and we will refer to the proofs there.  For an integer $m$, let $Q_m$ be the $m$th cyclotomic extension of $Q$.  (I.e., the extension of the rational field by the $m$th roots of unity.)  We write $|G| = ab$ where $a$ is a $\pi$-number and $b$ is a $\pi'$-number.  On page 50 of \cite{Isaacs2018}, Isaacs notes that ${\rm Gal} (Q_{ab})$ contains an automorphism $\sigma$ that acts like complex conjugation on $Q_b$ and fixes $Q_a$.  He calls $\sigma$ the {\it magic field} automorphism.  It is well known that the irreducible characters of $G$ have values lying in $Q_{ab}$.  On page 133 of \cite{Isaacs2018}, Isaacs shows that the magic field automorphism fixes every character in ${\rm B}_\pi (G)$.   Isaacs proves when $2 \in \pi$ or $|G|$ is odd that if $\chi^o \in {\rm I}_\pi (G)$ and $\chi$ is fixed by $\sigma$, then $\chi \in {\rm B}_\pi (G)$ (Theorem 5.2 of \cite{Isaacs2018}).

\section{Generalized Gallagher's theorem when $2 \in \pi$ or $|G|$ is odd}\label{secn: galla}

In this section let $\pi$ be a set of primes and let $G$ be a $\pi$-separable group.  We now consider the case where either $2 \in \pi$ or $|G|$ is odd.  In this situation, we can prove the generalized version of Gallagher's theorem.  We first prove a version of the generalized Gallagher's theorem for ${\rm B}_\pi$-characters.  When $N$ is a normal subgroup of $G$ and $\nu \in \irr (N)$, we write $\irr (G \mid \nu)$ for the set of irreducible constituents of $\nu^G$ and when $\nu \in {\rm B}_\pi (N)$, we set ${\rm B}_\pi (G \mid \nu) = {\rm B}_\pi (G) \cap \irr (G \mid \nu)$.


\begin{lem}\label{bpi-pre Gall,2inpi,|G|odd}
Let $\pi$ be a set of primes, let $G$ be a $\pi$-separable group, and let $N$ be a normal subgroup of $G$.  Suppose $2 \in \pi$ or  $|G|$ is odd.  Assume there exist characters $\varphi, \theta \in {\rm B}_{\pi} (N)$ so that $\varphi$ is $G$-invariant, $\theta$ extends to $\chi \in {\rm B}_{\pi} (G)$, and $\varphi \theta \in {\rm B}_\pi (N)$.    Then the map $\gamma \mapsto \gamma \chi$ is a bijection from ${\rm B}_\pi (G \mid \varphi) \rightarrow {\rm B}_\pi (G \mid \varphi\theta)$.
\end{lem}

\begin{proof}
We begin by noting since $\varphi \theta$ is irreducible and $\varphi$ and $\theta$ are $G$-invariant and $\theta$ extends to $\chi$ that we may apply Theorem 6.16 of \cite{Isaacs1976} to see that the map $\gamma \mapsto \gamma \chi$ is a bijection from $\irr (G \mid \varphi) \rightarrow \irr (G \mid \varphi\theta)$.

We first prove that if $\gamma \in {\rm B}_\pi (G \mid \varphi)$, then $\gamma \chi \in {\rm B}_\pi (G \mid \varphi\theta)$.  We do this via induction on $|G|$.  If $G = N$, then $\gamma = \varphi$ and $\chi = \theta$, and the result is trivial.  Thus, we may assume that $N < G$, and we may find $M$ maximal normal in $G$ so that $N \le M < G$.  Let $\mu = \chi_M$ so that $\mu_N = \theta$.  Take $\delta$ to be a constituent of $\gamma_M$.  We see that $\delta \in \irr (M \mid \varphi)$.  By Corollary 4.21 (a) of \cite{Isaacs2018},
we have that $\delta, \mu \in {\rm B}_\pi (M)$.  Thus, by the inductive hypothesis, we know that $\delta \mu \in {\rm B}_{\pi} (M)$.  Observe that by Problem 5.3 of \cite{Isaacs1976}, $\gamma \chi$ is a constituent of $(\delta\mu)^G$.

We know $G/M$ is either a $\pi$-group or a $\pi'$-group.  If $G/M$ is a $\pi$-group, then every constituent of $(\delta\mu)^G$ lies in ${\rm B}_\pi (G)$ by Theorem 4.14 of \cite{Isaacs2018}, so $\gamma \chi \in {\rm B}_\pi (G)$.  Thus, we may assume that $G/M$ is a $\pi'$-group.  Since $M$ is maximal normal and $2 \not\in \pi'$, we have that $|G/M| = p$ for some prime $p$ not in $\pi$.  
Observe that either $\delta^G$ is irreducible or $\delta$ extends to $G$.  If $\delta^G$ is irreducible, then $(\delta\mu)^G$ is irreducible by Problem 5.3 of \cite{Isaacs1976}.  Thus, $\gamma \chi = (\delta\mu)^G$.  By Theorem 4.19 of \cite{Isaacs2018}, $\gamma\chi \in {\rm B}_\pi (G)$.  Notice that if $\delta$ extends to $G$, then every constituent of $\delta^G$ is an extension.  Hence, $\gamma$ is an extension of $\delta$.  We have $(\gamma\chi)_M =\gamma_M \chi_M = \delta\mu$.  Thus, $((\gamma\chi)^o)_M = (\gamma_M\chi_M)^o = \delta^o\mu^o =(\delta\mu)^o$ is irreducible.  Thus, $(\gamma\chi)^o$ is irreducible.  The magic field automorphism fixes both $\gamma$ and $\chi$, so it fixes $\gamma \chi$.  Applying Theorem 5.2 of \cite{Isaacs2018}, we see that $\gamma\chi \in {\rm B}_\pi (G)$.

We have shown that the map $\gamma \mapsto \gamma \chi$ maps ${\rm B}_{\pi} (G \mid \varphi)$ to ${\rm B}_{\pi} (G \mid \varphi \theta)$.  We can use the initial observation on $\irr (G \mid \varphi)$ and $\irr (G \mid \varphi\theta)$ to see that this map must be an injection.

Suppose $\zeta \in {\rm B}_\pi (G \mid \varphi\theta)$.  We know that there exists $\kappa \in \irr ({G \mid \varphi})$ such that $\kappa \chi = \zeta$.  Since $\zeta$ is fixed by the magic field automorphism (as mentioned above this is proved on page 133 of \cite{Isaacs2018}), the bijection will imply that $\kappa$ is fixed by the magic field automorphism.  Since $\zeta^o = (\kappa \chi)^o = \kappa^o \chi^o$ is irreducible, it follows that $\kappa^o$ is irreducible.  By Theorem 5.2 of \cite{Isaacs2018}, we have $\kappa \in {\rm B}_\pi (G)$, and this implies that $\kappa \in {\rm B}_\pi (G \mid \varphi )$, as desired.  We deduce that the map is surjective, and this proves the lemma.
\end{proof}

We obtain Gallagher's theorem for ${\rm B}_\pi$-characters when either $2 \in \pi$ or $|G|$ is odd as an immediate corollary.

\begin{cor} \label{bpi Gall,2inpi,|G|odd}
Let $\pi$ be a set of primes, let $G$ be a $\pi$-separable group, and let $N$ be a normal subgroup of $G$.  Suppose $2 \in \pi$ or $|G|$ is odd.  Suppose there exists a character $\chi \in {\rm B}_{\pi} (G)$ so that $\theta = \chi_N \in \irr (N)$.  Then the map $\beta \mapsto \beta\chi$ is a bijection from ${\rm B}_{\pi} (G/N)$ to ${\rm B}_\pi (G \mid \theta)$.
\end{cor}

\begin{proof}
This is the special case of Lemma \ref{bpi-pre Gall,2inpi,|G|odd} with $\varphi = 1$.
\end{proof}

In the situation where $G$ is a group with a normal subgroup $N$ and $\eta \in {\rm I}_\pi (N)$, we define ${\rm I}_\pi (G \mid \eta)$ to be the partial characters in ${\rm I}_\pi (G)$ whose restriction to $N$ contains $\eta$.  When $\nu \in {\rm B}_\pi (N)$ so that $\eta = \nu^o \in {\rm I}_\pi (N)$, we do not see that it has been proved before that the map $\chi \mapsto \chi^o$ is a bijection from ${\rm B}_\pi (G \mid \nu) \rightarrow {\rm I}_\pi (G \mid \eta)$.  We prove that next.  However, as we mentioned above, there have been other sets of lifts constructed, and we realized that all we needed was that we had the bijection on $G$ and $N$ and that the constituent for the characters in the subset of $G$ lie in $N$.  Hence, we prove the result in that generality, and we will be able to apply this to other sets of lifts.

\begin{lem} \label{normal lifts}
Let $\pi$ be a set of primes, let $G$ be a $\pi$-separable group, and let $N$ be a normal subgroup of $G$.   Assume there exist subset $X_\pi (G) \subseteq \irr (G)$ and $X_\pi (N) \subseteq \irr (N)$ so that the map $\theta \mapsto \theta^o$ is a bijection from $X_\pi (G)$ to ${\rm I}_\pi (G)$ and from $X_\pi (N)$ to ${\rm I}_\pi (N)$, and assume for $\chi \in X_\pi (G)$ that the constituents of $\chi_N$ lie in $X_\pi (N)$.  If there exist $\nu \in X_\pi (N)$ and $\eta = \nu^o \in {\rm I}_\pi (N)$, then the map $\chi \mapsto \chi^o$ is a bijection from $X_{\pi} (G \mid \nu) \rightarrow {\rm I}_\pi (G \mid \eta)$.
\end{lem}

\begin{proof}
Suppose $\chi \in X_{\pi} (G \mid \nu)$.  We know that $\chi^o \in {\rm I}_\pi (G)$.  Also, $\chi$ is a constituent of $\nu^G$, so $\chi^o$ is a constituent of $(\nu^G)^o = (\nu^o)^G = \eta^G$, and so, $\chi^o \in {\rm I}_\pi (G \mid \eta)$.  Thus, the map is well-defined.  Since the map from $X_\pi (G)$ to ${\rm I}_\pi (G)$ is an injection, we see that the restriction will be an injection also.  Suppose $\zeta \in {\rm I}_\pi (G \mid \eta)$.  Let $\gamma \in X_\pi (G)$ so that $\gamma^o = \zeta$.  We know that $\eta$ is a constituent of $(\gamma_N)^o = (\gamma^o)_N = \zeta_N$.  We know that all the constituents of $\gamma_N$ lie in $X_\pi (N)$.  Since $\nu$ is the unique lift of $\eta$ in $X_\pi (N)$, it follows that $\nu$ is a constituent of $\gamma_N$.  Thus, $\gamma \in X_\pi (G \mid \nu)$.  It follows that the map is surjective, and we have the desired bijection.
\end{proof}

We now consider $\pi$-partial characters.  We obtain the generalized version of Gallagher's theorem when $2 \in \pi$ or $|G|$ is odd.  In other words, we prove Theorem \ref{ipi pre,2inpi,|G|odd}.


\begin{proof}[Proof of Theorem \ref{ipi pre,2inpi,|G|odd}]
Let $\gamma, \theta \in {\rm B}_\pi (N)$ so that $\gamma^o = \eta$ and $\theta^o = \xi$.  Take $\chi \in {\rm B}_\pi (G)$ so that $\chi^o = \zeta$.  Observe that since $\eta$ is $G$-invariant, it follows that $\gamma$ is $G$-invariant.  We have that $(\chi_N)^o = (\chi^o)_N = \zeta_N = \xi$.  We know that $\chi_N$ lies in ${\rm B}_\pi (N)$ and thus, $\chi_N = \theta$.  We may now apply Lemma \ref{bpi-pre Gall,2inpi,|G|odd} to see that there is a bijection from ${\rm B}_\pi (G \mid \varphi)$ to ${\rm B}_\pi (G \mid \varphi\theta)$.  It is not difficult to see that restriction is a bijection between ${\rm B}_\pi (G \mid \varphi)$ to ${\rm I}_\pi (G \mid \eta)$ and ${\rm B}_\pi (G \mid \varphi \theta)$ to ${\rm I}_\pi (G \mid \eta \xi)$ composing the appropriate versions of these bijections gives the desired bijection.
\end{proof}

We now obtain a version of Gallagher's theorem for $\pi$-partial characters when $2 \in \pi$ or $|G|$ is odd.  Notice that this Theorem \ref{ipi Gall,gen} with the additional assumption that $2 \in \pi$ or $|G|$ is odd.

\begin{cor}\label{ipi Gall,2inpi,|G|odd}
Let $\pi$ be a set of primes, let $G$ be a $\pi$-separable group, and let $N$ be a normal subgroup of $G$.  Suppose $2 \in \pi$ or $|G|$ is odd.  Suppose there exists a character $\zeta \in {\rm I}_{\pi} (G)$ so that $\xi = \zeta_N \in {\rm I}_\pi (N)$.  Then the map $\kappa \mapsto \kappa\zeta$ is a bijection from ${\rm I}_{\pi} (G/N)$ to ${\rm I}_\pi (G \mid \xi)$.
\end{cor}

\begin{proof}
This is the special case of Theorem \ref{ipi pre,2inpi,|G|odd} with $\eta = 1$.
\end{proof}

\section{Constructing nuclei from chains of subgroups} \label{secn: nucl}

As we stated in Section \ref{secn: pi theory}, we use a different set of lifts to prove the general version of the analog of Clifford's theory for $\pi$-partial characters.  We will be using the ideas from \cite{nucl}.  Given a group $G$, {\it a chain of normal subgroups} ${\mathcal N} = \{ G = N_0 > N_1 > \dots > N_n = 1 \}$ where each $N_i$ is normal in $G$.  Given a character $\chi \in \irr (G)$, write $\nu_0 = \chi$.  For $i \ge 0$, take $\nu_{i+1}$ to be an irreducible constituent of $(\nu_i)_{N_{i+1}}$.  Take $U$ to be the stabilizer of the set $\{ \nu_0, \nu_1, \dots, \nu_n \}$.  In Lemma 3.2 of \cite{nucl}, we prove that this sequence of characters determines a unique character $\phi \in \irr (U)$ and $\phi^G = \chi$.  We will say that $(U,\phi)$ is the ${\mathcal N}$-nucleus of $\chi$.  We note that all the choices for the $\nu_i$'s are conjugate, and so, a different choice of $\nu_i$'s will give a conjugate of $(U,\phi)$.  In any case, $(U,\phi)$ is uniquely determined up to conjugacy.

To define the lifts, we need $\pi$-special characters.  The $\pi$-special characters are initially defined by Gajendragadkar in \cite{gajen}.  We will follow the presentation of $\pi$-special characters in \cite{Isaacs2018}.  We need $\pi$ to be a set of primes and $G$ to be a $\pi$-separable group.  Given a character $\chi \in \irr (G)$.   We say that $\chi$ is {\it $\pi$-special} if $\chi (1)$ is a $\pi$-number and for every subnormal subgroup $S$, the determinantal order of the constituents of $\chi_S$ are $\pi$-numbers.  One interesting property of $\pi$-special characters is that if $\alpha$ is a $\pi$-special character of $G$ and $\beta$ is a $\pi'$-special character of $G$, then $\alpha\beta$ is irreducible (Theorem 2.2 of \cite{Isaacs2018}).  We say $\chi \in \irr (G)$ is {\it ($\pi$)-factored} if $\chi$ can be written as the product of a $\pi$-special and $\pi'$-special character.

Now, suppose ${\mathcal N}$ is a chain of normal subgroups where each of the factor groups is a $\pi$-group or a $\pi'$-group.  In Lemma 4.1 of \cite{nucl} it is proved that if $(U,\phi)$ is the ${\mathcal N}$-nucleus for $\chi \in \irr (G)$, then $\phi$ is $\pi$-factored.  With this in mind, we define ${\rm B}_\pi (G: {\mathcal N})$ to be the characters $\chi \in \irr (N)$ whose ${\mathcal N}$-nucleus $(U,\phi)$ have that $\phi$ is $\pi$-special.  In Theorem B of \cite{nucl}, it is shown that $\chi \mapsto \chi^o$ is a bijection from ${\rm B}_\pi (G : {\mathcal N}) \rightarrow {\rm I}_\pi (G)$.  If $N \in {\mathcal N}$ and ${\mathcal N}^*$ is the chain obtained by omitting the subgroups properly containing $N$, then it is shown in Theorem 5.2 (1) of \cite{nucl} for $\chi \in {\rm B}_\pi (G: {\mathcal N})$, the irreducible constituents of $\chi_N$ lie in ${\rm B}_\pi (N : {\mathcal N}^*)$.

\section{A general $\pi$ version of Gallagher's theorem} \label{secn: gen}

We now drop the hypothesis that $2 \in \pi$ or $|G|$ is odd.  However, in this general case, we have not been able to prove the generalized version of Gallagher's theorem.  Also, we have not been able to prove Gallagher's theorem for the ${\rm B}_\pi$ characters of $G$.  In order to obtain enough control of the nuclei we need to use the nuclei constructed in \cite{nucl} and the associated lifts there.  We start by obtaining a general result on the product of $\pi$-special characters.

\begin{lem} \label{pi spec prod}
Let $\pi$ be a set of primes.  If $\chi, \theta \in \irr (G)$ are $\pi$-special and $\chi\theta$ is irreducible, then $\chi\theta$ is $\pi$-special.
\end{lem}

\begin{proof}
Since $\chi (1)$ and $\theta (1)$ are $\pi$-numbers, it follows that $\chi(1) \theta (1)$ will be a $\pi$-number.  Let $S$ be a subnormal subgroup of $G$.  We know that $\chi_S$ and $\theta_S$ have irreducible constituents whose determinantal orders are $\pi$-numbers.  That is, if $\sigma$ is an irreducible constituent of $\chi_S$ and $\tau$ is an irreducible constituent $\theta_S$, then $\sigma$ and $\tau$ have determinantal orders that are $\pi$-numbers.  Now, if $\gamma$ is an irreducible constituent of $(\chi\theta)_S$, then there will exist $\sigma$ an irreducible constituent of $\chi_S$ and $\tau$ an irreducible constituent of $\theta_S$ so that $\gamma$ is a constituent of $\sigma \tau$.  Hence, the determinant of $\gamma$ will be a power of the determinant of $\sigma \tau$.  Now, computing with representations, one can show that
$${\rm det} (\sigma \tau) = {\rm det} (\sigma)^{\tau (1)} {\rm det} (\tau)^{\sigma (1)}.$$
Hence, it follows since $\sigma$ and $\tau$ have $\pi$-orders that $\sigma \tau$ will have $\pi$-order, and thus, $\gamma$ has $\pi$-order.
\end{proof}

We begin by considering the stabilizers of characters of normal subgroups contained in a normal subgroup where the character of the group restricts irreducibly.

\begin{lem} \label{ext stab}
Let $M < N < G$ be groups.  Suppose $\chi \in \irr (G)$ and $\nu = \chi_N$ is irreducible.  If $\theta \in \irr (M)$ is a constituent of $\nu_M$ and $T$ is the stabilizer of $\theta$ in $G$, then $G = TN$.
\end{lem}

\begin{proof}
We see that $\tau = \chi_{TN}$ is irreducible since $\chi_N$ is irreducible.  Observe that $T$ is the stabilizer of $\theta$ in $TN$ and since $\theta$ is a constituent of $\nu_M$, we see that $\theta$ is a constituent of $\tau_M = (\nu_N)_M$.   Let $\sigma \in \irr ({T \mid \theta})$ be the Clifford correspondent for $\tau$, so we have $\sigma^{TN} = \tau$.  We can also apply Clifford's theorem in $G$, and see that $\sigma^G$ is irreducible.  This implies that $\tau^G$ is irreducible.  The only way that $\chi_{TN} = \tau$ and $\tau^G$ irreducible can both be true is if $G = TN$.  Thus, we see that the lemma is proved.
\end{proof}

Using the previous lemma, we now compute the nucleus for a chain where a character restricts irreducibly to one of the subgroups in the chain.

\begin{lem} \label{ext nucl}
Let $G$ be a group and let ${\mathcal N}$ be a sequence of normal subgroups of $G$.  Assume $N \in {\mathcal N}$ and suppose $\chi \in \irr (G)$ satisfies that $\chi_N = \nu$ is irreducible.  If $(U,\varphi)$ is the ${\mathcal N}$-nucleus for $\chi$, then $G = NU$ and $\varphi_{N \cap U} \in \irr ({N \cap U})$.
\end{lem}

\begin{proof}
Let ${\mathcal N} = \{ G = N_0 > N_1 > \dots > N_n = 1 \}$.  Take $\theta_0 = \chi$ and for $0 \le i \le n-1$, take $\theta_{i+1}$ to be an irreducible constituent of $(\theta_i)_{N_{i+1}}$.  Let $T_i$ be the stabilizer in $G$ of $\theta_i$.  We know that $U = \cap_{i=1}^n T_i$.  Let $k$ be the integer so that $N_k = N$ and observe that for $0 \le i \le k$, we have that $\theta_i = \chi_{N_i}$ and so, $G = T_i$.  If every $\theta_i$ is invariant in $G$, then $G = U$ and $\chi = \varphi$, and the conclusion holds.

Thus, we may assume that there are some $\theta_i$'s that are not $G$-invariant.  Take $l$ to be the smallest integer so that $\theta_l$ is not $G$-invariant.  By the observation in the previous paragraph, we have that $l > k$, so $N_l < N$.   By Lemma \ref{ext stab}, we have that $G = N T_l$.   Let $\hat\chi \in \irr ({T_l \mid \theta_l})$ be the Clifford correspondent for $\chi$ with respect to $\theta_l$ and $\hat{\nu} \in \irr ({T_l \cap N \mid \theta_l})$ the Clifford correspondent for $\nu$.  We know that $\varphi^{T_l}$ is irreducible and $\theta_l$ is a constituent of $(\varphi^{T_l})_{N_l}$.  It follows that $\varphi^{T_l} = \hat\chi$.  We know that $\nu = \chi_N = ((\hat\chi)^G)_N = (\hat\chi_{N \cap T_l})^N$.  Since $\hat\chi_{N\cap T_l}$ induces irreducibly to $N$, it must irreducible, and so, $\hat\chi_{N \cap T_l} = \hat\nu$.  Set $\hat{\mathcal N}$ to be a sequence of $N_i \cap T_l$ for all $i$.  It is not difficult to see that $(U,\varphi)$ is the $\hat{\mathcal N}$-nucleus for $\hat\chi$.  By induction, $T_l = (N \cap T_l)U$ and $\varphi_{N \cap U}$ is irreducible.  We obtain $G = NT_l = N(N \cap T_l)U = NU$, and we have the desired conclusion.
\end{proof}

We use the previous lemma to compute the nucleus of a product when we have a character that restricts irreducibly to a normal subgroup with a character of the quotient.

\begin{lem} \label{prod nucl}
Let $G$ be a group and let ${\mathcal N}$ be a sequence of normal subgroups of $G$, and let $N \in {\mathcal N}$.  Suppose $\chi \in \irr (G)$ satisfies that $\chi_N = \nu$ is irreducible and $\gamma \in \irr ({G/N})$.  If $(U,\varphi)$ is the ${\mathcal N}$-nucleus for $\chi$ and $(V,\delta)$ is the ${\mathcal N}$-nucleus for $\gamma$, then $N \le V$ so $G = VU$, $\alpha = \varphi_{U \cap V} \in \irr ({U \cap V})$, $\beta = \delta_{U \cap V} \in \irr ({(U \cap V)/(U \cap N)})$, and $(U \cap V, \alpha \beta)$ is the ${\mathcal N}$-nucleus for $\chi \gamma \in \irr (G)$.
\end{lem}

\begin{proof}
Let ${\mathcal N} = \{ G = N_0 > N_1 > \dots > N_n = 1 \}$.   Let $k$ be the integer so that $N_k = N$ and observe that for $0 \le i \le k$, we have $N_k \le \ker {\gamma}$.  It is not difficult to see that $N \le V$.  By Lemma \ref{ext nucl}, we have $G  = NU$ and $\varphi_{N \cap U}$ is irreducible, so $G = VU$ and $\alpha = \varphi_{U \cap V}$ is irreducible.  Since $N \le V$, we have that $\beta = \delta_{V \cap U} \in \irr ({(U \cap V)/(U \cap N)})$. By Gallagher's theorem, we know that $\chi\gamma$ and $\alpha \beta$ are irreducible.

Take $\theta_0 = \chi$ and for $0 \le i \le n-1$, take $\theta_{i+1}$ to be an irreducible constituent of $(\theta_i)_{N_{i+1}}$ so that $\theta_k = \nu$.  Let $T_i$ be the stabilizer in $G$ of $\theta_i$.  Let $\sigma_0 = \gamma$, and for $0 \le i \le n-1$, let $\sigma_{i+1}$ be an irreducible constituent of $(\sigma_i)_{N_{i+1}}$.  For each $i$, set $S_i$ to be the stabilizer of $\sigma_i$ in $G$.  We know that $U = \cap_{i=1}^n T_i$ and $V = \cap_{i=1}^n S_i$.  Note for $i \le k$, we have $\sigma_i \in \irr (N_i/N)$ and $(\theta_i)_N = \nu$; so Gallagher's theorem applies and $\theta_i \sigma_i$ will be irreducible.  For $i \ge k$, we see that $\sigma_i = 1$, and again, we have that $\theta_i \sigma_i$ is irreducible.  Therefore, $\theta_i \sigma_i$ is irreducible for all $i$.  It is not difficult to see that its stabilizer will be $T_i \cap S_i$.  It follows that the subgroup for the ${\mathcal N}$-nucleus for $\chi \gamma$ is $\cap_{i=1}^n (T_i \cap S_i) = (\cap_{i=1}^n T_i) \cap (\cap_{i=1}^n S_i) = U \cap V$.  It is not difficult to determine that $\alpha \beta$ will be the character in $\irr ({U \cap V})$ so that $(U \cap V,\alpha\beta)$ is the ${\mathcal N}$-nucleus for $\chi \gamma$.
\end{proof}

Finally, we come to the version of Gallagher's theorem for characters in ${\rm B}_\pi (G:{\mathcal N})$.

\begin{lem} \label{bij bpi nucl}
Let $\pi$ be a set of primes, let $G$ be a $\pi$-separable group, let $N$ be a normal subgroup of $G$, and let ${\mathcal N}$ be a sequence of normal subgroups so that $N \in {\mathcal N}$ and all the factors are $\pi$ or $\pi'$-groups.   Suppose $\chi \in {\rm B}_{\pi} (G : {\mathcal N})$ so that $\nu = \chi_N \in \irr (N)$.  Then the map $\gamma \mapsto \gamma\chi$ is a bijection from
$${\rm B}_{\pi} (G/N : {\mathcal N}) \rightarrow {\rm B}_{\pi} (G \mid \nu : {\mathcal N}).$$
\end{lem}

\begin{proof}
By Gallagher's theorem, we know that $\gamma \mapsto \gamma \chi$ is a bijection from $\irr (G/N)$ $\rightarrow \irr (G \mid \nu)$.  Thus, it suffices to show that $\gamma \in {\rm B}_\pi (G/N : {\mathcal N})$ if and only if $\gamma \chi \in {\rm B}_{\pi} (G : {\mathcal N})$.

Let $(U,\varphi)$ be the ${\mathcal N}$-nucleus for $\chi$.  Since $\chi \in {\rm B}_{\pi} (G : {\mathcal N})$, we know that $\varphi$ is $\pi$-special.   Consider a character $\gamma \in \irr (G/N)$ and let $(V,\delta)$ be the ${\mathcal N}$-nucleus for $\gamma$.  We can write $\delta = \eta \zeta$ where $\eta$ is $\pi$-special and $\zeta$ is $\pi'$-special.  We know that $\gamma \in {\rm B}_\pi (G/N : {\mathcal N})$ if and only if $\zeta = 1$.  Since $\delta_{V \cap U}$ is irreducible, we see that $\tau = \eta_{V \cap U}$ and $\sigma = \zeta_{V \cap U}$ are irreducible.  Also, $\alpha = \gamma_{V \cap U}$ and $\tau$ will be $\pi$-special and $\sigma$ is $\pi'$-special.  Since $N \le \ker {\delta} = \ker {\eta} \cap \ker {\zeta}$, we see that $\sigma = 1$ if and only if $\zeta = 1$.

Since $\alpha \beta$ is irreducible and $\alpha \beta = \alpha \tau \sigma$ implies that $\alpha \tau$ is irreducible.  By Lemma \ref{pi spec prod}, $\alpha \tau$ is $\pi$-special.  Thus, $\gamma \chi \in {\rm B}_\pi (G : {\mathcal N})$ if and only if $\sigma = 1$.  We have already seen that $\sigma = 1$ if and only if $\zeta = 1$, and we have noted that $\zeta = 1$ if and only if $\gamma \in {\rm B}_\pi  (G/N : {\mathcal N})$.  This proves the result.
\end{proof}

We use the previous lemma to obtain a version of Gallagher's theorem for $\pi$-partial characters.


\begin{proof}[Proof of Theorem \ref{ipi Gall,gen}]
Take ${\mathcal N}$ to be a normal series for $G$ that contains $N$ so that the factor groups are $\pi$-groups or $\pi'$-groups.  Take $\chi \in {\rm B}_\pi (G : {\mathcal N})$ so that $\chi^o = \zeta$.  We know that $(\chi^o)_N = \zeta_N = \xi$ and so, $\chi_N$ is irreducible.  Take $\nu = \chi_N$.  We have bijections from ${\rm B}_\pi (G/N : {\mathcal N}) \rightarrow {\rm I}_{\pi} (G/N)$ and ${\rm B}_\pi (G \mid \nu : {\mathcal N}) \rightarrow {\rm I}_\pi (G \mid \xi)$.  Applying Lemma \ref{bij bpi nucl}, we have the bijection from ${\rm B}_\pi (G/N : {\mathcal N}) \rightarrow {\rm B}_\pi (G \mid \nu : {\mathcal N})$.  Composing these bijcetions appropriately, we obtain the desired bijection.
\end{proof}

\section{$\pi$-versions for a partial converse of Gallagher's theorem}

We now obtain a partial converse for the $\pi$-version of the converse of Gallagher's theorem.  We first prove the result for ${\rm I}_\pi$-characters.  Note that we do not need to assume $2 \in \pi$ or $|G|$ is odd to prove this direction.  We also note that a full converse of the $\pi$-versions of Gallagher's theorem is not possible, since if $|G:N|$ is a $\pi'$-number and $\eta \in {\rm I}_\pi (N)$, then we know that $\eta^G$ has a unique irreducible constituent $\zeta$ in ${\rm I}_{\pi} (G)$ in all cases and the principal partial character is the only partial character in ${\rm I}_\pi (G/N)$, so the map from ${\rm I}_\pi (G/N)$ to ${\rm I}_\pi ({G \mid \eta})$ will yield only $\zeta$ in all cases, and we will not be able to distinguish the case when $\eta$ extends to $G$.  Thus, some hypothesis on $G/N$ is needed.  Whether we can weaken it from being a $\pi$-group, we have not determined.


\begin{proof}[Proof of Theorem \ref{ipi conv}]
Let $\chi, \beta_1, \dots, \beta_n \in {\rm B}_\pi (G)$ so that $\chi^o = \zeta$ and $(\beta_i)^o = \kappa_i$ for $i = 1, \dots, n$.  Observe that $(\beta_i \chi)^o = \kappa_i \zeta$.  It follows that $\beta_1 \chi, \dots, \beta_n \chi$ are distinct and irreducible.  Hence, the hypotheses of Theorem \ref{main1} are met, and by that theorem, we see that $\chi_N$ is irreducible.  It follows that $\chi_N \in {\rm B}_\pi (N)$, and we conclude that $\zeta_N = (\chi^o)_N = (\chi_N)^o$ is irreducible, proving the theorem.
\end{proof}

In fact, the partial converse of Gallagher's theorem will apply for sets of lifts of the $\pi$-partial characters.  To see this we make the following definition we again only need that we have subsets of $\irr (G)$ and $\irr (N)$ that are in bijection with ${\rm I}_\pi (G)$ and ${\rm I}_\pi (N)$, and that irreducible constituents lie in this set.  Thus, we prove the result in this generality so that it could be applied to any set of lifts.  

\begin{cor}\label{lifts conv}
Let $\pi$ be a set of primes, let $G$ be a $\pi$-separable group, and let $N$ be a normal subgroup of $G$ so that $G/N$ is a $\pi$-group.  Assume there exist subsets $X_\pi (G) \subseteq \irr (G)$ and $X_\pi (N) \subseteq \irr (N)$ so that the map $\theta \mapsto \theta^o$ is a bijection from $X_\pi (G)$ to ${\rm I}_\pi (G)$ and from $X_\pi (N)$ to ${\rm I}_\pi (N)$, and assume for $\chi \in X_\pi (G)$ that the constituents of $\chi_N$ lie in $X_\pi (N)$.  If there exist $\chi \in X_\pi (G)$ and $\irr (G/N) = \{ \gamma_1 = 1, \dots, \gamma_n \}$ that satisfy $(\gamma_1 \chi)^o, \dots, (\gamma_n \chi)^o$ are irreducible and distinct, then $\chi_N \in X_\pi (N)$.
\end{cor}

\begin{proof}
Observe that the hypotheses of Theorem \ref{main1} are met, and so $\chi_N$ is irreducible.  The result now follows.
\end{proof}

Applying Corollary \ref{lifts conv}, if $G$ is $\pi$-separable, $N$ is a normal subgroup so that $G/N$ is a $\pi$-group  and either $\chi \in {\rm B}_\pi (G)$ or $\chi \in {\rm B}_\pi (G: {\mathcal N})$ for an appropriate normal series ${\mathcal N}$ so that the map $\beta \mapsto \chi \beta$ is a injection from $\irr (G/N)$ to $\irr (G)$, then $\chi_N$ will lie in ${\rm B}_\pi (G)$ or ${\rm B}_\pi (G: {\mathcal N}^*)$, respectively.  This yields a partial converse for the versions of Gallagher's theorem for ${\rm B}_\pi$-characters and ${\rm B}_\pi (G : {\mathcal N})$-characters.





\section*{Acknowledgments}
The first author thanks support of the program of Henan University of Technology (2024PYJH019),
the program of Foreign Experts of Henan Province (HNGD2024020),
and the Natural Science Foundation of Henan Province (252300421983).


\end{document}